\documentclass[a4paper,11pt,reqno]{article}
\usepackage{a4wide}
\usepackage[english]{babel}
\usepackage{amssymb}
\usepackage{amsmath}
\usepackage{amsfonts}
\usepackage{amsthm}

\usepackage{hyperref}
\hypersetup{colorlinks,citecolor=red,filecolor=purple,linkcolor=blue,urlcolor=black}
\usepackage{graphicx}

\usepackage{url}
\usepackage{enumerate}
\newcommand{\reto}{r_\star}
\newcommand{\beto}{b_\star}
\newcommand{\deto}{d_\star}
\newcommand{\Weto}{W_\star}
\newcommand{\Zeto}{Z_\star}

\newcommand{\R}{\mathbb R}
\newcommand{\pse}{\psi'({r})}
\newcommand{\Z}{\mathbb Z}
\newcommand{\Esp}{\mathbb E}

\newcommand{\p}{\mathbb P}
\newcommand{\1}[1]{\mathbf{1}\!_{\{#1\}}}
\newcommand{\EE}[1]{\Esp\left[#1\right]}
\newcommand{\ext}{\mathrm{Ext}}
\newcommand{\as}{\quad\mathrm{ a.s.}}
\newcommand{\eps}{\varepsilon}
\newcommand{\intpos}{\int_0^\infty}

\newcommand{\dif}{\mathrm{d}}

\newtheorem{prop}{Proposition}[section]

\newtheorem{thm}[prop]{Theorem}
\newtheorem{rem}[prop]{Remark}
\newtheorem{cor}[prop]{Corollary}
\newtheorem{conj}[prop]{Conjecture}

\begin{document}
\title{Birth and death processes with neutral mutations}
 \author{\textsc{Nicolas Champagnat$^{1}$, Amaury Lambert$^{2}$ and Mathieu Richard$^{2}$}}

\footnotetext[1]{Universit\'e de Lorraine, IECN, Campus Scientifique, B.P. 70239, Vand{\oe}uvre-l\`es-Nancy Cedex, F-54506, France
and Inria, Villers-l\`es-Nancy, F-54600, France E-mail: \href{mailto:Nicolas.Champagnat@inria.fr}{\nolinkurl{Nicolas.Champagnat@inria.fr}}
}

\footnotetext[2]{Laboratoire de Probabilit\'es et Mod\`eles Al\'eatoires,
UMR 7599 CNRS and UPMC Univ Paris 06,
Case courrier 188,
4 Place Jussieu,
F-75252 Paris Cedex 05, France, Emails:
\href{mailto:amaury.lambert@upmc.fr}{\nolinkurl{amaury.lambert@upmc.fr}}, \href{mailto:mathieu.richard@upmc.fr}{\nolinkurl{mathieu.richard@upmc.fr}}
}

\maketitle
\begin{abstract}In this paper, we review recent results of ours concerning branching processes with general 
lifetimes and neutral mutations, under the infinitely many alleles model, where mutations can occur either  at 
birth of particles or at a constant rate during their lives.

In both models, we study the allelic partition of the population at time $t$. We give closed-form formulae for 
the expected frequency spectrum at $t$ and prove pathwise convergence to an explicit limit, as $t\to+\infty$, 
of the relative numbers of types younger than some given age and carried by a given number of particles 
(small families). We also provide convergences in distribution of the sizes or ages of the largest families and of the oldest families.

In the case of exponential lifetimes, population dynamics are given by linear birth and death processes, and we can most of the time provide general formulations of our results unifying both models.
\end{abstract}

\tableofcontents

 \section{Introduction}
We consider a general branching model, where particles have i.i.d. (not necessarily exponential) life lengths and give birth at constant rate $b$ during their lives to independent copies of themselves. The genealogical tree thus produced  is called \emph{splitting tree} \cite{Geiger1996,Geiger1997,Amaury_contour_splitting_trees}. The process that counts the number of alive particles through time is a \emph{Crump-Mode-Jagers process} (or general branching process) \cite{Jagers_BP_with_bio} which is binary (births occur singly) and homogeneous (constant birth rate).

We enrich this genealogical model with mutations. In Model I, each child is a clone of her mother with probability $1-p$ and a mutant with probability $p$. In Model II, independently of other particles, each particle undergoes mutations during her life at  constant rate $\theta$ (and births are always clonal).
For both models, we are working under the infinitely many alleles model, that is, a mutation yields a type, also called \emph{allele}, which was never encountered before. Moreover, mutations are supposed to be neutral, that is, they do not modify the way particles die and reproduce. For any type and any time $t$, we call \emph{family} the set of all particles that share this type at time $t$.\\

Branching processes (and especially birth and death processes) with mutations have many applications in biology. In carcinogenesis \cite{Nowak09122003,IwasaApril2006,Sagitov2009,Durrett201042,Durrett_Mayberry,Durrett_Ovarian}, they can model the evolution of cancerous cells.
In \cite{Kendall1960}, Kendall  modeled carcinogenesis by a birth and death process where mutations occur during life according to an inhomogeneous Poisson process.
In \cite{Durrett_Ovarian,Durrett201042}, cancerous cells are modeled by a multitype branching process where a cell is of type $k$ if it has undergone $k$ mutations and where the more a cell has undergone mutations, the faster it grows. The object of study is the time $\tau_k$ of appearance of the first cell of type $k$.
In \cite{Sagitov2009}, the authors study the arrival time of the first resistant cell and the number of resistant cells, in a model of cancerous cells undergoing a medical treatment and becoming resistant after having experienced a certain number of mutations.

Branching processes with mutations are also used in epidemiology. Epidemics, and especially their onset, can be modeled by birth and death processes, where particles are infected hosts, births are disease transmissions and deaths are recoveries or actual deaths.
In \cite{Stadler05052011}, Stadler provides a statistical method for the inference of transmission rates and of the reproductive value of epidemics in a birth and death model with mutations.
In \cite{Lambert-Trapman}, Lambert \& Trapman enriched the transmission tree with Poissonian marks modeling detection events of hospital patients infected by an antibiotic-resistant pathogen. They provided an inference method based on the knowledge of times spent by patients at the hospital at the detection of the outbreak.

Let us also mention the existence of models, e.g. \cite{Gani_Yeo1965}, of phage reproduction within a bacterium by a (possibly time-inhomogeneous) birth and death process with Poissonian mutations, where particles model phage in the vegetative phase (DNA strands in the bacterium without protein coating) and death is interpreted as phage maturing (reception of protein coating).

In ecology, the neutral theory of biodiversity \cite{Hubbell} gives a prediction of the diversity patterns, in terms of species abundance distributions, that are generated by individual-based models where speciation is caused by mutation or by immigration from mainland. Usually, the underlying genealogical models are assumed to keep the population size constant through time, as in the Moran or Wright-Fisher models, and so have the same well-known properties as models in mathematical population genetics (e.g., Ewens sampling formula), with a different interpretation. See \cite{Haegeman-Etienne,Amaury-Immig-Mut} for cases where this assumption is relaxed in favor of the branching property.
\\

In this paper, we are first interested in the allelic partition of the population and more precisely in properties about the \emph{frequency spectrum} $(M_t^{i,a},i\geq1)$, where $M_t^{i,a}$ is the number of distinct types younger than $a$ (i.e., whose original mutation appeared after $t-a$) carried by exactly $i$ particles  at time $t$.
This kind of question was first studied by Ewens \cite{Ewens1972} who discovered the well known `sampling formula' named after him and which describes the law of the allelic partition for a Wright-Fisher model with neutral mutations.

In our models, it is not possible to obtain a counterpart of Ewens sampling formula but we obtain different kinds of results concerning the frequency spectrum $(M_t^{i,a},i\geq1)$. First, we get a closed-form formula for the expected frequency spectrum, even in the non-Markovian cases. Second, we get pathwise convergence results as $t\rightarrow +\infty$ on the survival event, of the relative abundances of types. Third, we investigate the order of magnitude of the sizes of the largest families at time $t$ and of the ages of oldest types at time $t$, as $t\rightarrow +\infty$, and show convergence in distribution of these quantities properly rescaled. Several regimes appear, depending on whether the \emph{clonal process}, which is the process counting particles of a same type, is subcritical, critical or supercritical.\\

We do not know of previous mathematical studies, other than ours, on branching processes with Poissonian mutations, but there are several existing mathematical results on branching models with mutations at birth that we now briefly review.

In discrete time, Griffiths and Pakes \cite{Griffiths1988a} studied the case of a Bienaym\'e-Galton-Watson (BGW) process where at each generation, all particles mutate independently with some probability $u$. The authors obtained properties about the number of alleles/types in the population, about the time of last mutation in the (sub)critical case and about the expected frequency spectrum. In \cite{Bertoin2009,Bertoin2010}, Bertoin considers an infinite alleles model with neutral mutations in a subcritical or critical BGW-process where particles independently give birth to a random number of clonal and mutant children according to the same joint distribution. In \cite{Bertoin2009}, the tree of alleles is studied, where all particles of a common type are gathered in clusters and the law of the allelic partition of the total population is given by describing the joint law of the sizes of the clusters and of the numbers of their mutant children. In \cite{Bertoin2010}, Bertoin obtains the joint 
convergence of the sizes of allelic families in the limit of large initial population size and small mutation rate.

In continuous time, Pakes \cite{Pakes1989} studied Markovian branching processes and gave the counterpart in the time-continuous setting, of properties found in the previously cited paper \cite{Griffiths1988a}. In particular, his results about the frequency spectrum and the ``limiting frequency spectrum'' are similar to ours, stated in Section \ref{small_families}.
Recently, Maruvka et al. \cite{Maruvka2010245,Maruvka2011} have considered the linear birth
and death process with Poissonian mutations.
Actually, they
rather studied a PDE satisfied by a concentration $n(x,t)$ which can be seen as (but is not
proved to be) a deterministic approximation to the number of families of size $x$ at time $t$.
It is remarkable that this PDE has a steady concentration $\bar{n}(x)$, whose behavior as
$x\to\infty$ is comparable to the asymptotic behavior of the relative numbers of families of size $m$ as $m\to+\infty$ in the discrete model studied here and in \cite{Griffiths1988a}. In
the monography \cite{Taib1992}, Ta\"ib is interested in general branching processes known as
Crump-Mode-Jagers processes (see \cite{Jagers_BP_with_bio,Jagers1984a} and references therein)
where mutations still occur at birth but with a probability that may depend for example on the
age of the mother. He obtained limit theorems about the frequency spectrum by using random
characteristics techniques but in most cases, limits cannot be explicitly computed. Some of
our results in Model I are applications of Ta\"ib's, but use techniques specific to splitting
trees to yield explicit formulae. We have refrained to apply results of Ta\"ib on the
convergence in distribution of properly rescaled sizes of largest families, on the validity of
which we have doubts in the case of supercritical clonal processes (see last section).

 The paper is organized as follows. In Section \ref{The models}, we define the models and give some of their properties that will be useful to state the main results. Section \ref{small_families} is devoted to the study of the frequency spectrum (small families). Finally, in Section \ref{Grandes_familles}, we give the results about ages of the oldest families and about sizes of the largest ones.

Notice that in this paper, most of the results are stated for linear birth and death processes in order to simplify the notation. Most of them are also true with general life length distributions and are proved in Chapter 3 of the PhD thesis \cite{TheseMathieu} for Model I, and in \cite{Champagnat2010,Champ_Lamb_2} for Model II. Specific effort has been put on finding a unifying formulation for our results as soon as it seemed possible.

 \section{The models}\label{The models}
 \subsection{Model without mutations}\label{without_mutation}
 We first define the model without mutations and give some of its properties. Afterwards, we will explain the two mutation mechanisms that we consider in this paper.

 As a population model, we consider \emph{splitting trees} \cite{Geiger1996,Geiger1997,Amaury_contour_splitting_trees}, that is,
\begin{itemize}
\item At time $t=0$, the population starts with one progenitor;
\item All particles have i.i.d. reproduction behaviors;
\item Conditional on her birth date $\alpha$ and her life length $\zeta$, each particle gives birth at a constant rate $b\in(0,\infty)$ during $(\alpha,\alpha+\zeta)$, to a single particle at each birth event.
\end{itemize}
It is important to notice that the common law of life lengths can be as general as possible.
Let $Z=(Z(t),t\geq0)$ be the process counting the number of extant particles through time. We denote the lifespan distribution by $\Lambda(\cdot)/b$ where $\Lambda$ is a finite positive measure on $(0,+\infty]$ with total mass $b$ and called \emph{lifespan measure} \cite{Amaury_contour_splitting_trees}.

The total population process $Z$ belongs to a large class of branching processes called \emph{Crump-Mode-Jagers} or \emph{CMJ processes}. In these processes, also called general branching processes~\cite{Jagers_BP_with_bio,Jagers1984a}, one associates with each particle $x$ in the population a non-negative r.v. $\lambda_x$ (her life length), and a point process $\xi_x$ called birth point process. One assumes that the sequence $(\lambda_x,\xi_x)_x$ is i.i.d. but $\lambda_x$ and $\xi_x$ are not necessarily independent.
 Then, the CMJ-process is defined as
 $$Z(t)=\sum_{x}\1{\sigma_x\leq t<\sigma_x+\lambda_x},\ \ t\geq0$$
 where for any particle $x$ in the population, $\sigma_x$ is her birth time.

 In our particular case, the common distribution of lifespans is $\Lambda(\cdot)/b$ and conditional on her lifespan, the birth point process of a particle is distributed as a Poisson point process during her life. We can say
that the CMJ-process $Z$ is \emph{homogeneous} (constant birth
rate) and \emph{binary} (births occur singly). We will say that $Z$ is \emph{subcritical, critical or supercritical} according to whether the mean number of children per particle 
\begin{equation}\label{defi_m}
m:=\intpos u\Lambda(\dif u) 
\end{equation}
is less than, equal to or greater than 1.

The advantage of homogeneous, binary CMJ-processes is that they enable explicit computations,
e.g., about one-dimensional marginals of $Z$ (see forthcoming Proposition
\ref{1d_marginal_Xi}).  More precisely, for $\lambda\geq0$, define
\begin{equation}
  \label{eq:def-psi}
  \psi(\lambda):=\lambda-\int_{(0,\infty]}(1-e^{-\lambda u})\Lambda(\dif u)
\end{equation}
and let $r$ be the greatest root of $\psi$. Notice that $\psi$ is convex, $\psi(0)=0$ and $\psi'(0)=1-m$. As a consequence,
\begin{equation}\label{prop_r}
\left\{\begin{array}{cl}                                                                                                     
                                                                                                      
r=0& \textrm{ if }Z \textrm{ is subcritical or critical,}\\
r>0&\textrm{ if }Z \textrm{ is supercritical.}
 \end{array}\right.                                                                                                 \end{equation}
Let $W$ be the so-called \emph{scale function} \cite[p.194]{Levy_processes} associated with $\psi$, that is, the unique increasing continuous function $(0,\infty)\rightarrow(0,\infty)$ satisfying
\begin{equation}\label{defi_W}
  \intpos W(x)e^{-\lambda x}\dif x=\frac{1}{\psi(\lambda)}, \quad \lambda>r.
\end{equation}

\begin{prop}[Lambert \cite{Amaury_contour_splitting_trees, Amaury-Immig-Mut}]\label{1d_marginal_Xi}
The one-dimensional marginals of $Z$ are given by
\begin{equation*}\label{loi_Xt}
  \p(Z(t)=0)=1-\frac{W'(t)}{bW(t)}
\end{equation*}
and for $n\ge 1$,
\begin{equation*}\label{loi_Xt_2}
  \p(Z(t)=n)=\left(1-\frac{1}{W(t)}\right)^{n-1}\frac{W'(t)}{bW(t)^2}.
\end{equation*}
 In other words, conditional on being non-zero, $Z(t)$ is distributed as a geometric r.v. with success probability $1/W(t)$.
 \end{prop}

If $\ext:=\left\{Z(t)\underset{t\rightarrow\infty}\longrightarrow 0\right\}$ denotes the extinction event of $Z$,
according to \cite{Amaury_contour_splitting_trees}, as a consequence of the last proposition,
 $$\p(\ext)=1-\frac{r}{b}.$$
 Thus, thanks to \eqref{prop_r}, extinction occurs a.s. when $Z$ is (sub)critical and
 $\p(\ext^c)>0$ when it is supercritical.
 
 The following proposition justifies the fact that $r$ is called the \emph{Malthusian parameter} of the population in the supercritical case.

 \begin{prop}[Lambert \cite{Amaury_contour_splitting_trees}] If $m>1$, conditional on the survival event $\ext^c$,
  \label{convergence_Amaury}

  \begin{equation}\label{conv_surcrit}e^{-r t}Z(t)\underset{t\rightarrow\infty}\longrightarrow \mathcal E\as
  \end{equation}
  where $\mathcal E$ is exponential with parameter $\pse$.
  \end{prop}
In fact, convergence in distribution is proved in \cite{Amaury_contour_splitting_trees} and a.s. convergence holds according to \cite{Nerman_supercritical_CMJ} (see \cite[p.285]{Splitting_trees_Immig}).

 \subsection{Two mutation models I and II}
We now assume that particles in the population carry  types, also called \emph{alleles}. We consider two population models where mutations appear in different ways. In each case, we will make the assumption of \emph{infinitely many alleles}, that is, to every mutation event is associated a different type, so that every type appears only once. We will also assume that mutations are \emph{neutral}, that is, they do not change the way particles die and reproduce.

 In Model I, mutations occur at birth. More precisely, there is some $p\in(0,1)$ such that at each birth event, independently of all other particles, the newborn is a clone of her mother with probability $p$ and a mutant with probability $1-p$. An illustration is given in Figure \ref{figure_spectre des frequencesI}.

 In Model II, particles independently experience mutations during their lives at constant rate $\theta>0$. In particular, in contrast with Model I, particles can change type several times during their lifetime, but always bear at birth the same type as their mother at this very time. An illustration is given in Figure \ref{figure_spectre des frequencesII}.\\
   \begin{figure}[!ht]
\begin{center}
\includegraphics{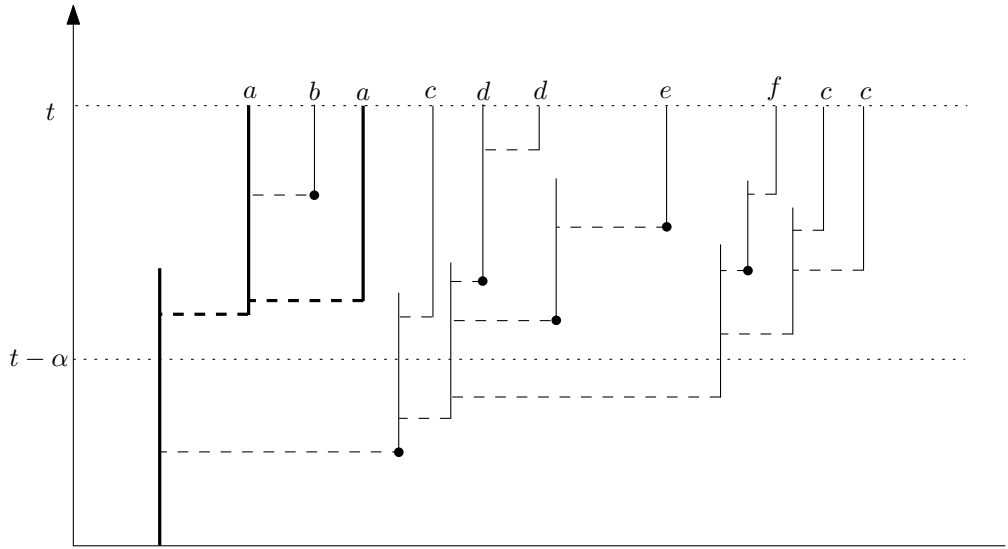}
\caption{An example of a splitting tree in Model I and of the allelic partition of the whole extant population at time $t$. Vertical axis is time and horizontal axis shows filiation (horizontal lines have zero length). Full circles represent mutations at birth and thick lines, the clonal splitting tree of the ancestor up to time $t$. The different letters are the alleles of alive particles at time $t$.}
\label{figure_spectre des frequencesI}
\end{center}
\end{figure}

    \begin{figure}[!ht]
\begin{center}
\includegraphics{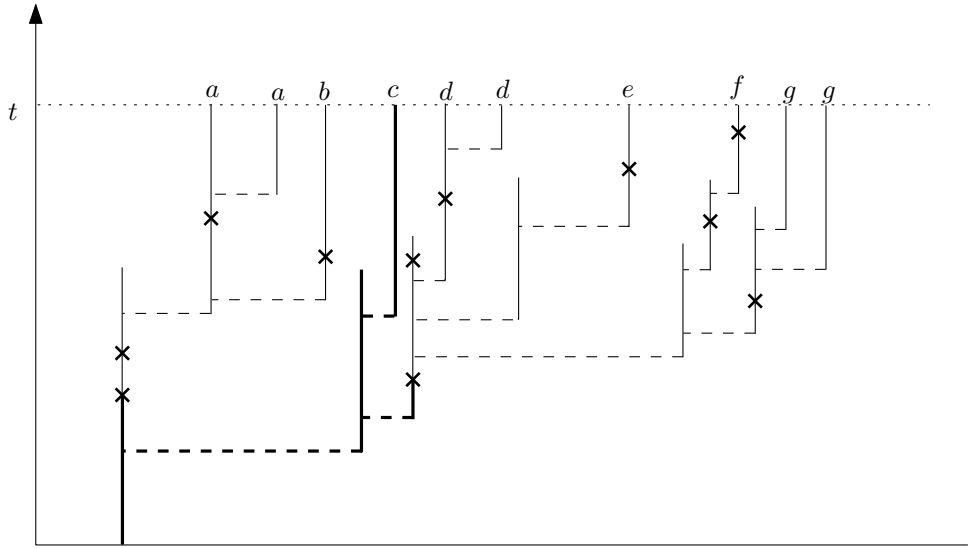}
\caption{An example of a splitting tree with mutations in Model II and of the allelic partition of the whole extant population at time $t$. Crosses represent mutations and thick lines, the clonal splitting tree of the ancestor up to time $t$. The different letters are the alleles of alive particles at time $t$. }
\label{figure_spectre des frequencesII}
\end{center}
\end{figure}

 In what follows, an important role will be played by the \emph{clonal process}, generically denoted $Z_\star$, counting, as time passes, the number of particles bearing the same type as the progenitor of the population at time 0.
It can easily be seen that the genealogy of a clonal population is again a splitting tree, so
that $Z_\star$ is also a homogeneous, binary CMJ process. We denote by $b_\star$ its birth
rate, by $\psi_\star$ the associated convex function as in \eqref{eq:def-psi} and by $\Weto$ the non-negative function with Laplace transform $1/\psi_\star$. Furthermore, when the clonal population is supercritical, \emph{i.e.} when $\psi_\star'(0+)<0$, we denote by $r_\star$ its Malthusian parameter, which is the only nonzero root of $\psi_\star$. We will sometimes need to have this generic notation depend on the model considered: $Z_p, \psi_p, W_p, r_p$ for Model I, and $Z_\theta, \psi_\theta, W_\theta, r_\theta$ for Model II.

Concerning Model I, it can be seen \cite{TheseMathieu} that the clonal splitting tree has the same life lengths as the original splitting tree and birth rate $b_p=b(1-p)$, so that its lifespan measure is $(1-p)\Lambda$ and
 $$\psi_p(\lambda)=p\lambda+(1-p)\psi(\lambda)\quad \lambda>0.$$
 In particular, as in \eqref{defi_m}, the clonal population is subcritical, critical or supercritical according to whether  $m(1-p)$ is less than, equal to or greater than 1. It should be noted that there is no closed-form formula  for $W_p$.

Concerning Model II, it can be seen \cite{Champagnat2010} that the clonal splitting tree has birth rate $b_\theta=b$ and life lengths distributed as $\min(X,Y)$ where $X$ has probality distribution $\Lambda(\cdot)/b$ and $Y$ is an independent exponential r.v. with parameter $\theta$. Then we get
 $$\psi_\theta(\lambda)=\frac{\lambda\psi(\lambda+\theta)}{\lambda+\theta}\quad\lambda>0.$$
In particular, $r_\theta=r-\theta$ and the clonal population is subcritical, critical or supercritical according to whether  $r$ is less than, equal to or greater than $\theta$. It can also be proved that $W$ and $W_\theta$ are differentiable and that their derivatives are related via 
 $$W_\theta'(x)=e^{-\theta x}W'(x)\quad x\geq0,
 $$
 with the requirement that $W_\theta(0)=1.$
\subsection{Exponential case}
An interesting  case that we will focus on is  the \emph{exponential (or Markovian) case}, when the common distribution of life lengths is exponential with parameter $d$ (with the convention that lifespans are a.s. infinite if $d=0$), that is,
$\Lambda(\dif u)=bde^{-du}\dif u$ or $\Lambda(\dif u)=b\delta_\infty(\dif u)$.
In that case, $Z$ is respectively a linear birth and death process with birth rate $b$ and death rate $d$ or a pure birth process (or Yule process) with parameter $b$.

In this case, $Z$ and $Z_\star$ are Markov processes and the quantities defined in Section \ref{without_mutation} are computable. Indeed, we have
\begin{gather}
\psi(\lambda)=\frac{\lambda(\lambda-b+d)}{\lambda+d},\quad r=b-d,\notag\\
m=1-\psi'(0)=\frac{b}{d}\ \textrm{ and }\ \pse=1-\frac{d}{b}.\label{m_et_psi_Expo}
\end{gather}
It is also possible to compute the function $W$, defined by \eqref{defi_W}, while it is generally unknown. From \cite[p.~393]{Amaury_contour_splitting_trees}, we have
\begin{equation*}\label{Wcexpo}
  W(x)=\left\{
  \begin{array}{cl}
  \frac{b e^{r x}-d}{r} &\textrm{ if }b\neq d\\
  1+b x &\textrm{ if }b=d
  \end{array}\right.\quad x\geq0
    \end{equation*}
 and in all cases
\begin{equation*}\label{deriveeWcexpo}
W'(x)=b e^{r x}\quad x\geq0.
\end{equation*}

The same results hold for $\Weto$,
by respectively replacing $b$, $d$ and $r$ by
\begin{subequations}
\renewcommand{\theequation}{\theparentequation-\Roman{equation}}
\begin{equation}
\beto:= b(1-p),\quad \deto:=d,\quad\reto:=r-bp
\end{equation}
in Model I and by
\begin{equation}
\beto:= b,\quad \deto:=d+\theta,\quad \reto:=r-\theta
\end{equation}
in Model II.
\end{subequations}

We will sometimes state results in the total generality of splitting trees, in which case an equation numbered ( -I) (resp. ( -II)) refers to Model I (resp. Model II), as done previously. However, we will most of the time focus on the exponential case, in which case we will as soon as possible use the unified notation using $\star$'s. We will notify when the results can be generalized and will give precise references.

\begin{rem}
In the exponential case, notice that Models I and II are two (incompatible) cases of a more general class of linear
birth and death processes with mutations, where particles mutate spontaneously at rate
$\theta$, die at rate $d$, give birth at rate $b$, and at each birth event: with probability
$p_2$, the mother and the daughter both mutate (and bear either the same new type or two
different new types); with probability $p_1$, the daughter (only) mutates; with probability
$p_0=1-p_1-p_2$, none of them mutates. Then Model I corresponds to the case when $\theta=p_2=0$ and Model II to the case when $p_1=p_2=0$. The case studied by Pakes in \cite{Pakes1989} corresponds to $\theta=0$, $p_0=u^2$, $p_1 = 2u(1-u)$ and $p_2= (1-u)^2$. It is still an open question to check whether, when our results hold  for both Models I and II with the unified notation, they hold for all linear birth and death processes with mutations.
\end{rem}

 \section{Small families}\label{small_families}
 Recall that a \emph{family} is a maximal set of particles bearing the same type at the same
 given time.
 In this section, we are interested in results about \emph{small families} that is, families whose sizes and ages are fixed, in opposition to those of Section \ref{Grandes_familles} which concern asymptotic properties of the largest and oldest ones.

More precisely, we give properties of the allelic partition of the entire population by studying the \emph{frequency spectrum} $(M_t^{i,a},i\geq1)$ where $M_t^{i,a}$ denotes the number of distinct types, whose ages are less than $a$ at time $t$, carried by exactly $i$ particles at time $t$.
Notice that $M_t^{i,t}$ is simply the number of alleles carried by $i$ particles at time $t$ (regardless of their ages).

 For instance, in Figure \ref{figure_spectre des frequencesI}, the frequency spectrum $(M_t^{i,t},i\geq1)$ is $(3,2,1,0,\dots)$ because three alleles ($B,E,F$) are carried by one particle, $A$ and $D$ are carried by two particles and $C$ is the only allele carried by three particles.
 Moreover, if we only consider families with ages less than $a$, $(M_t^{i,a},i\geq1)$ equals $(3,1,0,\dots)$ because alleles $A$ and $C$ appear in the population before time $t-a$.
 Similarly, in Figure \ref{figure_spectre des frequencesII}, the frequency spectrum in Model II is $(4,3,0,\dots)$.

In the case of branching processes, there is no closed-form formula available for the law of the frequency spectrum as it is the case for the Wright-Fisher model thanks to Ewens sampling formula \cite{Ewens1972}.
 Nevertheless, we obtained for both mutation models an exact computation of the expected frequency spectrum and almost sure asymptotic behavior of this frequency spectrum as $t\to+\infty$.

  \subsection{Expected frequency spectrum}\label{moyenne spectre allelique}
We first give an exact expression of the expected frequency spectrum at any time $t$.

For $0<a<t$ and $i\geq1$, we denote by $M_t^{i,\dif a}$ the number of types carried by $i$ particles at time $t$ and with ages in $[a-\dif a,a]$.
The following proposition yields its expected value.

\begin{prop}\label{wxcv}
For $0<a<t$ and $i\geq1$, we have
\begin{subequations}\label{expect_spectrum}
\renewcommand{\theequation}{\theparentequation-\Roman{equation}}
\begin{equation}\label{expect_spectrumI}
\Esp[M_t^{i,\dif a}]=\frac{p}{b(1-p)}W'(t-a)\left(1-\frac{1}{W_p(a)}\right)^{i-1}\frac{W_p'(a)}{W_p^2(a)}\ \dif a.\end{equation}
\begin{equation}\label{expect_spectrumII}
\Esp[M_t^{i,\dif a}]=\frac{\theta  W'(t)}{b}\left(1-\frac{1}{W_\theta(a)}\right)^{i-1}\frac{e^{-\theta a}}{W_\theta^2(a)}\dif a
\end{equation}
\end{subequations}
In the exponential case, both expressions read as
\begin{equation}\label{expect_spectrum_expo}
\Esp[M_t^{i,\dif a}]=(r-\reto)e^{rt}\left(1-\frac{1}{\Weto(a)}\right)^{i-1}\frac{e^{-(r-\reto)a}}{\Weto^2(a)}\dif a.
\end{equation}

\end{prop}

In \cite{TheseMathieu}, \eqref{expect_spectrumI} is proved in the general case. Its proof uses the branching property and basic properties about Poisson processes. 
The main argument is that conditional on $Z(t-a)$, $M_t^{i,\dif a}$ is the sum of $Z(t-a)$ independent r.v. distributed as 
the number of mutants that appear in the population in a time interval $da$ and with $i$ clonal alive descendants at time $a$.
The proof of the general case of \eqref{expect_spectrumII} in \cite{Champagnat2010} is based on coalescent point processes.

The expected frequency spectrums $\Esp[M_t^{i,a}]$ can be obtained by integrating \eqref{expect_spectrumI} and \eqref{expect_spectrumII} over ages. Taking into account the contribution of the type of the progenitor, we can prove the following result.
\begin{cor}\label{moyenne spectre allelique 2}
For $a\leq t$ and $i\geq1$,

\begin{subequations}\label{expect_spectrum2}
\renewcommand{\theequation}{\theparentequation-\Roman{equation}}
{\setlength\arraycolsep{2pt}
\begin{eqnarray}
\Esp[M_t^{i,a}]&=&\frac{p}{b(1-p)}\int_0^a W'(t-x)\left(1-\frac{1}{W_p(x)}\right)^{i-1}\frac{W_p'(x)}{W_p^2(x)}\dif x\nonumber\\
&&\qquad\qquad\qquad\qquad+\frac{1}{b(1-p)}\left(1-\frac{1}{W_p(t)}\right)^{i-1}\frac{W_p'(t)}{W_p^2(t)}\1{a=t}\label{expect_spectrum2I}.
\end{eqnarray}}
{\setlength\arraycolsep{2pt}
\begin{eqnarray}
\Esp[M_t^{i,a}]&=& \frac{\theta  W'(t)}{b}\int_0^a\left(1-\frac{1}{W_\theta(x)}\right)^{i-1}\frac{e^{-\theta x}}{W_\theta^2(x)}\dif x\nonumber\\
&&\qquad\qquad\qquad\qquad+W(t)\left(1-\frac{1}{W_\theta(t)}\right)^{i-1}\frac{e^{-\theta t}}{W_\theta^2(t)}\1{a=t}\label{expect_spectrum2II}.
\end{eqnarray}}
\end{subequations}
In the exponential case,
$$\Esp[M_t^{i,a}]=(r-\reto)e^{rt}\int_0^a\left(1-\frac{1}{\Weto(x)}\right)^{i-1}\frac{e^{-(r-\reto)x}}{\Weto^2(x)}\dif x+\p(\Zeto(t)=i)\1{a=t}$$
\end{cor}
The second terms that appear in the r.h.s. correspond to the probabilities that the progenitor has $i$ alive clonal descendants at time $t$. In the exponential case, we left this probability as such, since its expression depends on the model.
It is also possible to get similar equations for the number of families with ages less than $a$ (resp. with size $i$) by summing over $i$ (resp. by taking $a=t$) in the last expressions.

\begin{rem}In the exponential case, when the process $Z$ is critical, that is, when $r=b-d=0$, for $a<t$,
$$\Esp[M_t^{i,a}]=\frac{|\reto|}{\beto}\frac{1}{i}\left(1-\frac{1}{\Weto(a)}\right)^{i},$$ which is reminiscent of Fisher log-series of species abundances \cite{Amaury-Immig-Mut}.
Surprisingly, this expression is independent of $t\in(a,\infty)$.
\end{rem}

From Corollary \ref{moyenne spectre allelique 2}, we deduce the asymptotic behavior of $\Esp[M_t^{i,a}]$ in the supercritical case.
\begin{prop}\label{wowy}We suppose that $m>1$. In the general case,
\begin{equation}\label{oasis}
\lim_{t\to+\infty}e^{-r t}\Esp[M_t^{i,a}]= \frac{r}{b}\frac{1}{\pse}J^{i,a}
\end{equation}
where, for Model I,
\begin{subequations}\label{Jia}
\renewcommand{\theequation}{\theparentequation-\Roman{equation}}
\begin{equation}\label{JiaI}
J^{i,a}:=\frac{p}{(1-p)}\int_0^a \left(1-\frac{1}{W_p(u)}\right)^{i-1}\frac{e^{-r u}W_p'(u)}{W_p^2(u)}\dif u
\end{equation}
and, for Model II,
\begin{equation}\label{JiaII}
J^{i,a}:=\theta\int_0^a\left(1-\frac{1}{W_\theta(u)}\right)^{i-1}\frac{e^{-\theta u}}{W_\theta^2(u)}\dif u.
\end{equation}
\end{subequations}
In the exponential case, we get the simpler formula
\begin{equation*}
J^{i,a}=(r-\reto)\int_0^a\left(1-\frac{1}{\Weto(u)}\right)^{i-1}\frac{e^{-(r-\reto)u}}{\Weto^2(u)}\dif u.
\end{equation*}
\end{prop}
Notice that $\Esp[M_t^{i,a}]$ grows exponentially with parameter $r$, as does $Z$ on its survival event.

\subsection{Convergence results}\label{Convergence results}
In this section and in all following ones, we are interested in long-time behaviors in the two models we consider. Then, from now on, \textbf{we assume that the process $Z$ is supercritical}.

This paragraph deals with improvements of the convergence results \eqref{oasis} regarding the expected frequency spectrum. The following results yield the asymptotic behavior as $t\rightarrow +\infty$ of the frequency spectrum $(M_t^{i,a},i\geq1)$,
 conditional on the survival event.

 The main technique we use to prove them is CMJ-processes counted with random characteristics (see \cite{Jagers_BP_with_bio} and Appendix A in \cite{Taib1992}). It enables us to obtain several pathwise convergence results regarding some  processes embedded in the supercritical splitting tree.

 A characteristic is a random non-negative function on $[0,+\infty)$.
 To each particle $x$ in the population, is associated a characteristic $\chi_x$, which can be viewed as a score or a weight. It must satisfy  that $(\lambda_x,\zeta_x,\chi_x)_x$ is an i.i.d. sequence, where we recall that $\lambda_x$ is the life length of $x$ and $\zeta_x$ its birth process.
 Then, the process counted with the characteristic $\chi$ is defined as
 \begin{equation}\label{CMJwith_charac}
Z^\chi(t):=\sum_{x}\chi_x(t-\sigma_x)\1{\sigma_x\leq t}.
\end{equation}
 For instance, if $\chi(t)=\1{t\leq\lambda_x}$, $Z^\chi$ equals $Z$ and if $\chi(t)=\1{t\leq\lambda_x\wedge a}$, $Z^\chi(t)$ is the number of extant particles at time $t$ with ages less than $a$.
 Then, provided technical conditions about $\chi$ are satisfied, the convergences of
 $e^{-r t}Z^\chi(t)$ and of $Z^\chi(t)/Z(t)$ as $t\to+\infty$ hold a.s. on the survival event. In our case, when $\chi$ is appropriately chosen, we can use this result to obtain the following statements.

 \begin{prop}
  \label{nombre_mutants}

Let $M_t$ be the number of extant types at time $t$.
  Almost surely, on the survival event of $Z$,
\begin{gather}
\lim_{t\to+\infty}e^{-r t}M_t=
J \mathcal E\notag\\
\lim_{t\to+\infty}e^{-r t}M_t^{i,a}= J^{i,a} \mathcal E\label{silence}
%
\end{gather}
where in Model I,
\begin{subequations}\label{defi_J}
\renewcommand{\theequation}{\theparentequation-\Roman{equation}}
\begin{equation}\label{defi_JI}
J:=\frac{r p}{1-p}\intpos e^{-r u}\ln(W_p(u))\dif u,
\end{equation}
while in Model II,
\begin{equation}\label{defi_JII}
J:=\theta\intpos\frac{e^{-\theta x}}{W_\theta(x)}\dif x
\end{equation}
\end{subequations}
and where $\mathcal E$ is the r.v. defined by \eqref{conv_surcrit}. 

In the exponential case, we have
$$J=(r-\reto)\int_0^\infty\frac{e^{-(r-\reto)u}}{\Weto(u)}\dif u.$$
\end{prop}
Notice that \eqref{silence} is consistent with \eqref{oasis} since $\p(\ext^c)=r/b$ and $\Esp[\mathcal E]=1/\pse$.
Moreover, \eqref{silence} still holds after $M_t^{i,a}$ is replaced by $M_t^{i,t}$ and $J^{i,a}$ by $J^{i,\infty}$.\\

 \subsection{Asymptotic behavior of the limiting frequency spectrum}\label{Behavior_limiting_frequency spectrum}
  Thanks to Proposition \ref{nombre_mutants}, the proportion $M_t^{i,a}/M_t$ of types carried by $i$ particles and with ages less than $a$ converges a.s. to $J^{i,a}/J$ as $t\to+\infty$. This limit is called ``the limiting frequency spectrum'' by Pakes in \cite{Pakes1989}. This paragraph is devoted to the asymptotic behavior, as $i\to+\infty$, of $J^i:=J^{i,\infty}$, obtained by taking $a=\infty$ in \eqref{JiaI} and \eqref{JiaII}. In the exponential case,
  \begin{equation}\label{Jiaetoile}
   J^{i}=(r-\reto)\int_0^\infty\left(1-\frac{1}{\Weto(u)}\right)^{i-1}\frac{e^{-(r-\reto)u}}{\Weto^2(u)}\dif u.
  \end{equation}

\subsubsection{Supercritical case}
In this paragraph, we only treat the exponential case.
Let us assume that the clonal process is supercritical, that is, $r_\star>0$.
 Define $$\nu:=\frac{r}{\reto},\quad \mu:=\frac{\beto}{\reto},\quad \gamma:=\frac{r-\reto}{\beto}.$$
 We have $\gamma=p/(1-p)$ in Model I and $\gamma=\theta/b$ in Model II. Recall that $J^i$ is the proportion of types carried by $i$ particles in the large time asymptotic. 

 \begin{prop}\label{equivalent_supercritic}
 In the exponential case, we have for both models
 $$J^i\underset{i\to+\infty}\sim i^{-1-\nu}\gamma\Gamma(\nu+1)\mu^\nu.$$
\end{prop}
Notice that this result is consistent with \cite{Maruvka2011} where Maruvka et al. use an
approximation of the frequency spectrum by a concentration driven by a PDE, and with
\cite{Pakes1989} where Pakes considers Markov branching processes with multiple simultaneous
births, binomial mutations at birth and no Poissonian mutations.

\begin{rem}
The following proof of Proposition \ref{equivalent_supercritic} easily extends to any life length distributions since it is based on Proposition \ref{convergence_Amaury} which holds in the general case.
 \end{rem}

\begin{proof}[Proof of Proposition \ref{equivalent_supercritic}]
 Since $\Weto(t)\geq1$ for $t\geq0$, the sequence $\left(J^i\right)_{i\geq1}$ is positive and non-increasing. Then, according to a Tauberian theorem about series, to prove Proposition \ref{equivalent_supercritic}, it is sufficient to prove that $\sum_{i\geq j}J^i$ is equivalent to
 $i^{-\nu}\gamma\Gamma(\nu)\mu^\nu$ as $j\to+\infty$.

 Recalling \eqref{Jiaetoile}, we have $$\sum_{i\geq j}J^i=(r-\reto)\intpos e^{-r t}\p(\Zeto(t)\geq j)\dif t$$ and from now on, we follow the proof of \cite[Thm 3.2.1]{Pakes1989}.
Let $s>0$ be such that $j=e^{\reto s}$. Then $j^\nu=e^{r s}$ and
{\setlength\arraycolsep{2pt}
\begin{eqnarray*}
j^\nu\int_{0}^\infty e^{-r t}\p(\Zeto(t)\geq j)\dif t&=&\int_{-s}^\infty e^{-r t}\p(\Zeto(t+s)\geq e^{\reto s})\dif t\\
&=&\int_{-s}^\infty e^{-r t}\p(e^{-\reto(t+s)}\Zeto(t+s)\geq e^{-\reto t})\dif t.
\end{eqnarray*}}
Using Proposition \ref{convergence_Amaury} and \eqref{m_et_psi_Expo},  $\p(e^{-\reto(t+s)}\Zeto(t+s)\geq e^{-\reto t})\underset{s\to+\infty}\longrightarrow \p(E_\star\geq e^{-\reto t})$ where $E_\star$ is a non-negative r.v. such that $$\p(E_\star=0)=1-\frac{\reto}{\beto}$$
and conditional on $\{E_\star>0\}$, $E_\star$ is an exponential r.v. with parameter $\psi_\star'(\reto)=1-\frac{\deto}{\beto}=\frac{1}{\mu}$.
Moreover, using Markov inequality, for $\eps>0$,
$$\p(e^{-\reto(t+s)}\Zeto(t+s)\geq e^{-\reto t})\leq e^{\reto(\nu+\eps)t}\EE{\left(\Zeto(s+t)e^{-\reto(s+t)}\right)^{\nu+\eps}}\leq C e^{\reto(\nu+\eps)t}
$$
using again the a.s. convergence in Theorem \ref{convergence_Amaury}. Then, for $s>0$ and $t\in\R$, we have
$$e^{-r t}\p(e^{-\reto(t+s)}\Zeto(t+s)\geq e^{-\reto t})\leq e^{-r t}\1{t>0}+C e^{\reto\eps t}\1{t<0}$$
and thanks to the dominated convergence theorem,
$$j^\nu\int_{0}^\infty e^{-r t}\p(\Zeto(t)\geq j)\dif t\underset{j\to+\infty}\longrightarrow \int_{\R} e^{-r t}\p(E_\star\geq e^{-\reto t})\dif t=\frac{1}{\mu}\int_{\R} e^{-r t}e^{-\frac{1}{\mu}e^{-\reto t}}\dif t$$
 The change of variables $x=e^{-\reto t}$ in the last integral leads to
 $$j^\nu\int_{0}^\infty e^{-r t}\p(\Zeto(t)\geq j)\dif t\underset{j\to+\infty}\longrightarrow\frac{1}{\beto}\int_0^{\infty}x^{\nu-1}e^{-x/\mu}\dif x=\frac{\Gamma(\nu)}{\beto}\mu^\nu,$$
 which terminates the proof.
 \end{proof}
\subsubsection{Critical case}
We want to obtain a similar result to Proposition \ref{equivalent_supercritic} when the clonal population is critical.
It seems that this is not possible in a general setting due to the non explicit expression of the functions $W_p$ and $W_\theta$. However, in the exponential and critical case, we have the simpler expression $\Weto(t)=1+\beto t$ and $\reto=0$.
  Then, we have
 $$J^i=r\intpos e^{-rt}\left(1-\frac{1}{1+\beto t}\right)^{i-1}\frac{1}{(1+\beto t)^2}\dif t.$$
 \begin{prop}\label{equivalent_critic}
In the exponential case, we have
 \begin{equation*}\label{comportem_Ji}
 J^i_\star\underset{i\to+\infty}\sim C\left(\gamma\right) i^{-3/4}e^{-2\sqrt{\gamma i}}
 \end{equation*}
where we recall that here $\gamma=r/\beto=r/d_\star$ and we have set $C(x)=\sqrt{\pi}e^{x/2}x^{5/4}$ for $x>0$.
 \end{prop}

 \begin{proof}
  By a change of variables, we have set
  $$J^i=\frac{r}{\beto}\intpos e^{rt/\beto}\frac{t^{i-1}}{(1+t)^{i+1}}\dif t=\gamma\Gamma(i)U(i,0,\gamma)$$ where $U$ is known as a confluent hypergeometric function (see \cite[Ch.13]{MR0167642}).
  Then, using \cite[Thm 3.3.2]{Pakes1989} with $B=-2$, we have the result.
 \end{proof}


  \section{Asymptotic results about large and old families}\label{Grandes_familles}
 We now state results about ages of the oldest families and about sizes of the largest ones. We mainly focus on the case  when clonal populations are subcritical. Then, in Subsection \ref{Other_results}, we explain which results hold in the critical and supercritical cases.

 We need some notation.
 For $t\geq0$,
 \begin{itemize}
  \item for $a\geq0$, let $O_t(a)$ be the number of extant families at time $t$, with ages greater than $a$ ($O$ for ``old''); for convenience, we set $O_t(a)=0$ if $a<0$.
  \item for $x\in\R$, let $L_t(x)$ be the number of families with sizes greater than $x$ at time $t$ ($L$ for ``large'').
 \end{itemize}
In this section, we are interested in finding the orders of magnitudes of the ages and of the sizes of the families, that is, in finding numbers $c_t$ and $x_t$ such that $\EE{O_t(c_t)}$ and $\EE{L_t(x_t)}$ converge to positive and finite real numbers as $t\to+\infty$.

 \subsection{Ages of old families in the subcritical case}\label{typical ages}
In this section, we suppose that the clonal processes are subcritical and we are interested in ages of old families.
Although we only state the results in the exponential case, they also hold in the general case and are proved in \cite[Ch. 3]{TheseMathieu} and \cite{Champ_Lamb_2}. However, to obtain the general results in Model I, additional assumptions about the lifespan measure $\Lambda$ are required, which are easily satisfied in the exponential case (for instance, we need the existence of a negative root of $\psi_p$, which, with easy computations, is $b(1-p)-d$ in the exponential case).

In the first result, which is a result in expectation, we show that in both models, the ages are of order of magnitude $$c_t:=\frac{r}{r-\reto}t.$$

 \begin{prop}[\cite{TheseMathieu,Champ_Lamb_2}]\label{expect subcritical}
  We suppose that $\Zeto$ is subcritical. For $a\in\R$, we have
$$\Esp[O_t(a+c_t)]\underset{t\rightarrow\infty}\longrightarrow\frac{|\reto|}{\deto}
e^{-(r-\reto) a}.$$
\end{prop}
 This result is a consequence of the expected spectrum formula \eqref{expect_spectrum_expo}, summed over $i\geq1$ and integrated on $(a+c_t,t)$. We also obtain a more precise result about the convergence in distribution of $O_t(a+c_t)$ as $t\to+\infty$.
\begin{prop}[\cite{TheseMathieu,Champ_Lamb_2}]\label{critical_ages}
With the same assumptions as in Proposition \ref{expect subcritical}, for $a\in\R$,
conditional on the survival event, as $t\to\infty$,
$O_t(a+c_t)$ converges in distribution to a r.v. $\mathcal O$, distributed as a mixed Poisson r.v. whose parameter of mixture is
$$\frac{b}{r}\frac{|\reto|}{\deto}\ e^{-(r-\reto)a} E.$$
where $E$ is an exponential r.v. with mean 1. Equivalently, $\mathcal O$ is geometric on $\{0,1,\cdots\}$ with success probability 
$$\frac{1}{1+\frac{b}{r}\frac{|\reto|}{\deto}\ e^{-(r-\reto)a}}.$$
\end{prop}
 The proof of this proposition in the general case and for Model I, given in \cite{TheseMathieu}, follows arguments of Ta\"ib in \cite{Taib1992} and
  uses the notion of CMJ processes counted with \emph{time-dependent} random characteristics developed by Jagers and Nerman in \cite{Jagers1984a,Jagers1984}. The difference with \eqref{CMJwith_charac} is that here the characteristics are allowed to depend on time.
 This theory provides convergences in distribution, as $t\to+\infty$, of quantities of the form
 $$Z^{\chi^t}_t:=\sum_{x}\chi^t_{x}(t-\sigma_x)\1{\sigma_x\leq t},$$
 under technical conditions about the family of characteristics $(\chi^t(\cdot),t\geq0)$. The
 proof of Proposition \ref{critical_ages} for Model II is given in~\cite{Champ_Lamb_2} and
 does not make use of random characteristics.\\

 The last result deals with the convergence in distribution of the sequence of ranked ages of extant families.
 Let $\mathcal M(\R)$ be the set of non-negative $\sigma$-finite measures on $\R$ and finite on $\R^+$,
equipped with the \emph{left-vague topology} induced by the maps $\nu\mapsto\int_{\R} f(x)\nu(\dif x)$
for all bounded continuous functions $f$ such that there exists $x_0\in\R$ satisfying $\forall x\leq x_0, f(x)=0$.

 \begin{thm}[\cite{TheseMathieu,Champ_Lamb_2}]\label{point process2}With the same assumptions as previously,
let $X_t$ be the point process defined by
$$X_t(\dif x):=\sum_{k\geq1}\delta_{A_t^k-c_t}(\dif x)$$ where $A_t^1\geq A_t^2\geq\cdots$ is the decreasing sequence of ages of alive families at $t$.
Then, conditional on the survival event, $X_t$ converges as $t\rightarrow\infty$ in $\mathcal M(\R)$ equipped with the left-vague topology to a mixed Poisson point process with intensity measure
\begin{equation*}\label{mesure_carac2}
\frac{b}{r}\frac{|\reto|}{\deto}E\ (r-\reto)e^{-(r-\reto)x}\dif x
\end{equation*}
where $E$ is an exponential r.v. with mean 1.
\end{thm}

  \subsection{Sizes of largest families in the subcritical case of Model II}\label{sizes}
 In this paragraph, we still suppose that the clonal process is subcritical and we are interested in similar results as those of Subsection \ref{typical ages} about the sizes of the largest families. The aim is to find a number $x_t$ such that $L_t(x_t)$ converges to a finite and positive limit as $t\to+\infty$.

 Concerning Model I, this problem is still open. 
  On the contrary, it is possible to obtain in Model II the sizes of the largest families. In \cite{Champ_Lamb_2}, they are given for any life length distribution but to simplify the results, we only state them in the exponential case. The following result is a consequence of \eqref{expect_spectrum2II} applied with $a=t$ and summed over $i\geq x_t+c$. Recall that the clonal process is assumed to be subcritical, so that $\theta>r$.

 \begin{prop}[\cite{Champ_Lamb_2}]
 We set  $$x_t:=\frac{r t-\frac{\theta}{\theta-r}\log t}{-\log\left(\frac{b}{\theta+d}\right)}.$$
 Then, for $c\in\R$,
  $$\Esp\left[L_t(x_t+c)\right]\underset{t\to+\infty}\sim A(b,d,\theta)\left(\frac{b}{\theta+d}\right)^{c-1+\{-x_t-c\}}$$
  where $\{x\}$ denotes the fractional part of a real number $x$ and where $A(b,d,\theta)$ is an explicit constant that only depends on $b$, $d$ and $\theta$.
 \end{prop}

 For $t\geq0$ and $k\geq1$, we denote by $S^{k}_t$ the size of the $k$-th largest family in the whole population at time $t$.
Let $$X_t:=\sum_{k\geq1}\delta_{S_t^{k}-x_t}$$ be the point measure of the renormalized sizes of the population.
To get rid of fractional parts, the following theorem gives convergence in distribution of $L_t(x_t+c)$ and $X_t$  along a subsequence.
More precisely, for $n\geq 1$, let $t_n$ be such that $x_{t_n}=n$; this equation has a unique solution for any $n$ greater than some integer $n_0$. It satisfies
$$t_n\underset{n\to+\infty}\sim \frac{\theta-r}{\theta}\log\left(\frac{\theta+d}{b}\right)n.$$
We now state the convergence of the sequence $(X_{t_n},n\geq n_0)$.

 \begin{thm}[\cite{Champ_Lamb_2}]\label{pointprocess_size_II}
Conditional on the survival event, the sequence $(X_{t_n},n\geq n_0)$ of point processes on $\Z$
converges as $n\to+\infty$ on the set $\mathcal M(\R)$ equipped with the left-vague topology to a
mixed Poisson point measure on $Z$ with intensity measure
$$A(b,d,\theta)\frac{b-d}{b}\left(1-\frac{b}{\theta+d}\right) E\ \sum_{k\in\Z}\left(\frac{b}{\theta+d}\right)^{k-1}\delta_k$$
where the mixture coefficient $E$ is an exponential r.v. with mean $1$.
 \end{thm}

  \subsection{Other results}\label{Other_results}
  \subsubsection{Critical case in Model I}
  The case of a critical clonal process $Z_p$ for a general supercritical splitting tree is treated in Section 3.5.1 of \cite{TheseMathieu} where the counterparts of Propositions \ref{expect subcritical} and \ref{critical_ages} and Theorem \ref{point process2} are proved.

  If $(1-p)m=1$, provided that the second moment $\sigma^2:=\intpos \Lambda(\dif u)u^2$ is finite and that a condition about the tail distribution of $\Lambda$ holds, ages of oldest families are of order
  $$c_t=t-\frac{\log t}{r}.$$
  Notice that these conditions about $\Lambda$ are trivially satisfied in the exponential case.
  These results were also proved in \cite[Ch. 4]{Taib1992} for any CMJ-process $Z$, \emph{i.e.} with a birth point process as general as possible, but in that case, limits were not explicit.

  Similarly to the subcritical case, the problem of sizes of the largest families is still open.
  Nevertheless, we can state the following conjecture about their order of magnitude.
  \begin{conj} \label{conjecture}
   If $$x_t:=\frac{2}{\sigma^2}(r t^2-t\log t),$$ as $t\to\infty$, on the survival event, $L_t(x_t)$ converges in distribution to a non-degenerate geometric r.v. 
  \end{conj}

  \subsubsection{Critical case in Model II}
  The general case when $Z_\theta$ is critical ($\theta=r$) can be found in Sections 3.4 and 5 in \cite{Champ_Lamb_2}.
  For both ages and sizes, the counterparts of the results of Subsections \ref{typical ages} and \ref{sizes} hold.

  As in Model I, ages of the oldest families are of order $c_t=t-\frac{\log t}{r}.$
Moreover, sizes of the largest ones are of order $$x_t=\frac{r^2}{4\pse}\left(t-\frac{\log t}{2r}\right)^2$$
and the point measure
$$\sum_{k\geq1}\delta_{\sqrt{S_t^{k}}-\sqrt{x_t}}$$
converges to a mixed Poisson measure as $t\to+\infty$ but contrary to Theorem \ref{pointprocess_size_II}, it does not only hold along a subsequence.

  \subsubsection{Sizes of largest families in supercritical cases}
 In \cite[Ch.~3]{TheseMathieu}, general splitting trees in Model I are considered.
 When the clonal process $Z_p$ is supercritical, that is, when $(1-p)m>1$, a result about the
 sizes of the largest families is proved.
 First notice that, as in \eqref{conv_surcrit},
$$\p(Z_p(t)\longrightarrow0)=1-\frac{r_p}{(1-p)b}$$
and on $\{Z_p(t)\longrightarrow0\}^\mathrm{c}$,
$e^{-r_p t}Z_p(t)$ a.s. converges as $t\rightarrow\infty$ to an exponential random variable.
Hence, the sizes of alive families at time $t$ must be of order $e^{r_p t}$ as $t\to+\infty$.
We proved this in \cite{TheseMathieu} by showing that $\EE{L_t\left(e^{r_p t}\right)}$ converges as $t\to+\infty$ to an explicit limit.

Notice that we cannot obtain similar results to Proposition \ref{critical_ages} and Theorem \ref{point process2} concerning the convergence in distribution of $L_t\left(e^{(b(1-p)-d)t}\right)$ and
the convergence of the associated point measure of the decreasing sequence of family sizes.

In \cite{Taib1992}, Ta\"ib considers a more general model than our Model I; mutation mechanism is the same but $Z_p$ can be any supercritical CMJ-process. In his Theorem 4.6, by using a time-dependent characteristic argument, he proved the convergence in distribution of $L_t\left(e^{r_p t}\right)$ (to a non-explicit random variable).
However, we have doubts about the application of Theorem A.7, since the technical requirements
of this theorem do not seem to hold in his case. These technical requirements are neither
proved to hold in \cite{Taib1992} nor in \cite{Jagers1984}.\\

In Model II, for a general supercritical splitting tree, if $Z_\theta$ is supercritical, that is, $r>\theta$, $Z_\theta(t)$ asymptotically grows like $e^{(r-\theta)t}$.
In \cite[Prop. 3.2]{Champ_Lamb_2}, it is proved that $\EE{L_t\left(e^{(r-\theta)t}\right)}$ converges as $t\to+\infty$, but we were unable to obtain any convergence in distribution in that case.

 \paragraph{Acknowledgments.}
This work was supported by project MANEGE ANR-09-BLAN-0215
(French national research agency). The authors want to thank an anonymous referee for his/her careful check of this manuscript.

\end{document}